\documentclass{amsart}
\usepackage{amsfonts,amsmath,color,graphicx,calligra,mathrsfs,tikz,tikz-cd,mathtools}
\usepackage{amsthm,amssymb}

\usetikzlibrary{matrix,calc,arrows}
\tikzcdset{scale cd/.style={every label/.append style={scale=#1},
		cells={nodes={scale=#1}}}}

\usepackage[english]{babel}
\usepackage{comment}
\usepackage{hyperref}

\newtheorem{thm}{Theorem}[section]
\newtheorem{cor}[thm]{Corollary}

\newtheorem{prop}[thm]{Proposition}
\theoremstyle{definition}
\newtheorem{defn}[thm]{Definition}
\theoremstyle{remark}
\newtheorem{rem}[thm]{Remark}
\theoremstyle{definition}

\theoremstyle{definition}

\theoremstyle{remark}

\numberwithin{equation}{section}
\newcommand{\qqq}{\mathfrak{q}}

\newcommand{\Tor}{\mathrm{Tor}}
\newcommand{\Ext}{\mathrm{Ext}}

\newcommand{\coker}{\mathrm{coker}}
\newcommand{\im}{\mathrm{im}}
\newcommand{\HHom}{\mathrm{Hom}}
\newcommand{\HH}{\mathrm{H}}
\newcommand{\Hochs}{\mathrm{HH}}
\newcommand{\eqdef}{\mathrel{\mathop:}=}
\newcommand{\Symm}{\mathrm{S}}
\newcommand{\Ex}{\mathrm{Ex}}
\newcommand{\Exal}{\mathrm{Exalcom}}
\newcommand{\pdim}{\mathrm{pd}}
\newcommand{\idim}{\mathrm{id}}
\newcommand{\xMapsto}[2][]{\ext@arrow 0599{\Mapstofill@}{#1}{#2}}

\def\Mapstofill@{\arrowfill@{\Mapstochar\Relbar}\Relbar\Rightarrow}

\makeatother

\makeatletter
\def\@tocline#1#2#3#4#5#6#7{\relax
	\ifnum #1>\c@tocdepth % then omit
	\else
	\par \addpenalty\@secpenalty\addvspace{#2}%
	\begingroup \hyphenpenalty\@M
	\@ifempty{#4}{%
		\@tempdima\csname r@tocindent\number#1\endcsname\relax
	}{%
		\@tempdima#4\relax
	}%
	\parindent\z@ \leftskip#3\relax \advance\leftskip\@tempdima\relax
	\rightskip\@pnumwidth plus4em \parfillskip-\@pnumwidth
	#5\leavevmode\hskip-\@tempdima
	\ifcase #1
	\or\or \hskip 1em \or \hskip 2em \else \hskip 3em \fi%
	#6\nobreak\relax
	\hfill\hbox to\@pnumwidth{\@tocpagenum{#7}}\par% <---- \dotfill -> \hfill
	\nobreak
	\endgroup
	\fi}
\makeatother

\begin{document}

\title[]{A ring of cohomological operators on Ext and Tor}%
\author{Samuel Alvite and Javier Majadas}%
\address{Departamento de Matem\'aticas, Facultad de Matem\'aticas, Universidad de Santiago de Compostela, E15782 Santiago de Compostela, Spain}%
\email{samuel.alvite.pazo@usc.es, j.majadas@usc.es}%

\thanks{$^{(\star)}$ This work was partially supported by Agencia Estatal de Investigaci\'on (Spain), grant PID2020-
115155GB-I00 (European FEDER support included, UE) and by Xunta de Galicia through the Competitive Reference Groups (GRC) ED431C 2023/31. Samuel Alvite was also financially supported by Xunta de Galicia Scholarship ED481A-2023-032.}

\keywords{Gulliksen operators, Andr\'e-Quillen cohomology}%
\thanks{2020 {\em Mathematics Subject Classification.} Primary: 13D03, 13D07. Secondary: 13D10.}

\begin{abstract}
  Let $f \colon R \to B$ be a surjective homomorphism of rings with kernel $I$. Gulliksen (when $I$ is generated by a regular sequence) and later Mehta (in general) showed that for any $B$-modules $M$ and $N$, $\Ext_B^{\ast}(M,N)$ has a structure of graded $\Symm_B^{\ast}\left(\widehat{I/I^2}\right)$-module, where $\widehat{\mkern9mu}$ denotes dual and $\Symm$ denotes symmetric algebra. This construction is extended to the case where $f$ is not necessarily surjective in a way that allows one to regard these operators from a more natural perspective.
\end{abstract}

\maketitle
% ----------------------------------------------------------------
\setcounter{secnumdepth}{1}
\setcounter{section}{-1}
\tableofcontents

\section{Introduction}

Unless otherwise stated, all rings in this paper will be commutative (in the graded case sometimes anticommutative).

Let $R \to B$ be a surjective ring homomorphism with kernel $I$ generated by a regular sequence of length $n$ and $M, N$ a pair of $B$-modules. Gulliksen \cite{Gulliksen} showed that $\Ext_B^{\ast}(M,N)$ has a structure of graded $B[X_1, \dots, X_n]$-module, where $B[X_1, \dots, X_n]$ is the graded polynomial ring over $B$ on variables $X_1, \dots, X_n$ of degree 2. He also showed that if $R$ is noetherian and $M, N$ $B$-modules of finite type such that $\pdim_R(M) < \infty$ or $\idim_R(N) < \infty$, then $\Ext_B^{\ast}(M,N)$ is a $B[X_1, \dots, X_n]$-module of finite type.

Two years later, Mehta \cite{Mehta} gave another construction for these operators which allows them to be defined for an arbitrary surjective ring homomorphism $R \to B$, showing that $\Ext_B^{\ast}(M,N)$ has a structure of graded $\Symm_B^{\ast}\left(\widehat{I/I^2}\right)$-module, where $I= \ker(R \to B)$, $\widehat{I/I^2} \eqdef \HHom_B(I/I^2,B)=\Ext_R^1(B,B)$, and $\Symm^{\ast}_B$ denotes the symmetric algebra. When $I$ is generated by a regular sequence, $\Symm_B^{\ast}\left(\widehat{I/I^2}\right) = B[X_1, \dots, X_n]$ and the structure agrees with the one given by Gulliksen. Subsequent constructions were given by Eisenbud \cite{Eisenbud}, Avramov \cite{Avramov}, Avramov-Sun \cite{AvramovSun}, etc. Fifty years after Gulliksen's construction, these operators continue being extensively used.

If $R$ is an algebra over a field $k$, when $I$ is generated by a regular sequence, Snashall and Solberg \cite[Section 7]{SnaSol} show that the structure of $\Symm_B^{\ast}\left(\widehat{I/I^2}\right)$-module factorizes by the canonical action of Hochschild cohomology $\Hochs^{\ast}(B | k )$ on \break $\Ext_B^{\ast}(M,N)$. Note that, under these hypotheses, if $\HH^{\ast}(R,B,W)$ denotes Andr\'e-Quillen cohomology \cite{An1974} \cite{Quillen}, $\widehat{I/I^2} \cong \HH^1(R,B,B)$ and, analyzing the proof given in \cite{SnaSol}, one can see that in this case the action of $\Symm_B^{\ast}\left(\widehat{I/I^2}\right)$ factorizes through $\Symm_B^{\ast}\left(\HH^1(k,B,B)\right)$, via the canonical homomorphisms $\HH^1(R,B,B) \to \HH^1(k,B,B)$ and $\HH^1(k,B,B) \to \Hochs^2(B | k)$ defined in \cite{Quillen}.

In this paper we first extend Mehta's construction to the case of a not necessarily surjective ring homomorphism $A \to B$ and show that (at least when $A$ is a field) it also agrees, in this general setting, with the operations of $\Symm_B^{\ast}\left(\HH^1(A,B,B)\right)$ given by the canonical homomorphism $\HH^1(A,B,B) \to \Hochs^2(B | A)$.

However, we think that the main interest of extending these operations to the non-surjective case is not the greater generality, but rather the fact that the construction is more natural and that even in the surjective case, $R \to B$, we may factorize the operations by the ones corresponding to some homomorphism $A \to B$ where $R$ is an $A$-algebra. We will make this precise below.

For any ring homomorphism $A \to B$ and $B$-modules $M$ and $N$, we construct an homomorphism of $B$-modules with central image (Definition \ref{homPsi}, Proposition \ref{ProofPsi}, Proposition \ref{central})
\[ \psi \colon \HH^1(A,B,B) \to \Ext_B^2(M,M) \]
giving a structure (via the Yoneda product by $\psi(x)$) of $\Symm_B^{\ast}(\HH^1(A,B,B))$-algebra on $\Ext_B^{\ast}(M,N)$.

We show (Theorem \ref{commPhiPsi}) that when $R$ is an $A$-algebra and $B=R/I$, this map makes the following diagram commutative
\begin{center}
	\begin{tikzcd}[scale cd=0.93, sep=small]
		\mathrm{Der}_A(R,M) \arrow{r} & \HH^1(R,B,B)=\widehat{I/I^2} \arrow{rr} \arrow[swap]{rd}{\varphi} & & \HH^1(A,B,B) \arrow{r} \arrow{ld}{\psi} & \HH^1(A,R,B)
		\\ & & \Ext_B^2(M,M) & &  	
	\end{tikzcd}
\end{center}
where the upper row is the Jacobi-Zariski exact sequence associated to $A \to R \to B$ \cite[5.1]{An1974}, and where $\varphi$ is the homomorphism defining Mehta's structure. Thus, in particular, when $A \to B$ is surjective, taking $R=A$ we see that our definition recovers that of Mehta's. In general, our definition is very similar to his and, in fact, even easier for some purposes, since both definitions involve a connecting homomorphism, but in the case of $\psi$ it is an isomorphism.

We also prove (Theorem \ref{commSnashallSolberg}) that (when $A$ is a field) our construction agrees with the known action through Hochschild cohomology. This extends the case proved in \cite{SnaSol}, of a surjective homomorphism with kernel generated by a regular sequence, to the case of an arbitrary ring homomorphism, but it also provides a computation of the action given from Andr\'e-Quillen cohomology through Hochschild cohomology, obtaining as a result an easy description of this action: it agrees with our Definition \ref{homPsi}. The simplicity of this definition is one of the reasons why we believe that $\HH^1(A,B,B)$ can naturally be regarded as the source of Gulliksen operators.

Other advantages of considering the operators defined by an arbitrary homomorphism $A \to B$ are the following. On the one hand, when $R \to B$ is surjective with kernel $I$ generated by a regular sequence, Gulliksen showed the aforementioned finiteness result. This result translates (Theorem \ref{finiteGeneration}) to the arbitrary case $A \to B$ in a natural way: the hypotheses are now that the homomorphism $A \to B$ is complete intersection.

On the other hand, even when we are interested in the case $B=R/I$ with $I$ generated by a regular sequence, the action of $B[X_1, \dots, X_n]$ on $\Ext_B^{\ast}(M,N)$ can be far from being faithful. An obvious example is when $B$ is the residue field of a regular local ring $R$. Thus, it is useful to factorize the ring of cohomological operators $\Symm_B^{\ast}\left(\widehat{I/I^2}\right)$ by its quotient by the annihilator of $\Ext_B^{\ast}(M,N)$. This leads to the study of support varieties or schemes (see e.g. \cite{AvramovBuchweitz}). If we can choose a ring $A$ so that $R$ is an $A$-algebra and $\HH^1(A,R,B)=0$, the diagram above shows that $\Symm_B^{\ast}(\HH^1(A,B,B))$ is a quotient of $\Symm_B^{\ast}\left(\widehat{I/I^2}\right)$, avoiding some of the trivial action of $\Symm_B^{\ast}\left(\widehat{I/I^2}\right)$ on $\Ext_B^{\ast}(M,N)$ (note, however, that given $R$ and $I$, we cannot always expect to find such a ring $A$ so that the action of $\Symm_B^{\ast}(\HH^1(A,B,B))$ be faithful (Remark \ref{h2})). The same diagram also reveals part of that trivial action: if $\hat{x} \in  \widehat{I/I^2} = \HHom_R(I,B)$ is the restriction of a derivation $D \colon R \to B$, then $\varphi(\hat{x}) = \psi(0)=0$.

\section{Review on extensions of modules and algebras}

We review some basic facts about extensions of modules and algebras that we will use. Details can be seen in \cite[$0_{\rm IV}$, \S18, \S20]{EGAIV1} and \cite{BoA10} (\cite{MacLane} can also be useful).

\begin{prop} \label{Prop1.1}
	Let $B$ be a (commutative) ring, $\Ex$ a functor from the category of $B$-modules to the category of sets satisfying:
	\begin{enumerate}
		\item[(i)] For any pair of $B$-modules $M_1$, $M_2$, there exists a natural isomorphism
		\[ \tau \colon \Ex(M_1) \times \Ex(M_2) \to  \Ex(M_1 \oplus M_2)\]
		such that, if $p_i \colon M_1 \oplus M_2 \to M_i$ is the canonical projection, then  \break $\Ex(p_i) \circ \tau \colon \Ex(M_1) \times \Ex(M_2) \to  \Ex(M_i)$ is the canonical projection.
		\item[(ii)] $\Ex(0)$ is a singleton and the composition homomorphism
		\[ \Ex(M) \xrightarrow{\cong} \Ex(M) \times \Ex(0) \xrightarrow{\tau} \Ex(M) \]
		is the identity map.
	\end{enumerate}
	Then, for any $B$-module $M$, $\Ex(M)$ has a $B$-module structure.
\end{prop}
\begin{proof}
	Since $M$ is a $B$-module, one has $B$-module maps:
	\begin{enumerate}
		\item[(a)] $0 \to M$
		\item[(b)] $+ \colon M \times M \to M$; $(m,n) \mapsto m+n$
		\item[(c)] $- \colon M \to M $; $m \mapsto -m$
		\item[(d)] For every $b \in B$, $(b\cdot ) \colon M \to M$; $m \mapsto bm$
	\end{enumerate}
	along with certain commutative diagrams involving these maps, representing associativity, commutativity, etc.
	
	From the conditions (i) and (ii), applying $\Ex$ to these morphisms and their associated diagrams, one obtains the analogues for $\Ex(M)$, which is the desired result.
\end{proof}

\begin{defn} \label{defExModulos}
	Let $B$ be a ring, $n \geq 1$ an integer. An exact sequence of $B$-modules
	\[ \xi = \left( 0 \to N \to T_{n-1} \to \cdots \to T_0 \to M \to 0 \right)\]
	will be called an \emph{$n$-extension} of the module $M$ by the module $N$. If
	\[ \xi' = \left( 0 \to N \to T'_{n-1} \to \cdots \to T'_0 \to M \to 0 \right)\]
	is another $n$-extension of $M$ by $N$,  $\xi'$ is said to be \emph{directly related} to $\xi$ if there exists a commutative diagram of $B$-module maps
	\begin{center}
		\begin{tikzcd}
			0 \arrow{r} & N \arrow{r} \arrow[equal]{d} & T_{n-1} \arrow{r} \arrow{d}& \cdots \arrow{r} & T_0 \arrow{r} \arrow{d} & M \arrow{r} \arrow[equal]{d} & 0
			\\ 0 \arrow{r} & N \arrow{r} & T'_{n-1} \arrow{r} & \cdots \arrow{r} & T'_0 \arrow{r} & M \arrow{r} & 0
		\end{tikzcd}
	\end{center}
	
	Two extensions are said to be \emph{equivalent} if they are in the same equivalence class with respect to the the equivalence relation generated by being directly related (we shall not concern ourselves with any arising set-theoretic questions in this paper). The equivalence class of an extension $\xi$ will be denoted by $[\xi]$, while the set of all equivalence classes of $n$-extensions of $M$ by $N$ will be denoted by $\Ex_B^n(M,N)$.
\end{defn}

\begin{defn}
	Let $B$ be a ring,
	\[ \xi = \left( 0 \to N \to T_{n-1} \to \cdots \to T_0 \to M \to 0 \right)\]
	an $n$-extension of $M$ by $N$, $\alpha \colon N \to N'$ and $\beta \colon M' \to M$ $B$-module maps. We have commutative diagrams of exact sequences
	\begin{center}
		\begin{tikzcd}[column sep=0.81cm]
			0 \arrow{r} & N \arrow{r} \arrow{d}{\alpha} & T_{n-1} \arrow{r} \arrow{d} &  T_{n-2} \arrow{r} \arrow[equal]{d} & \cdots \arrow{r} & T_0 \arrow{r} \arrow[equal]{d} & M \arrow{r} \arrow[equal]{d} & 0
			\\ 0 \arrow{r} & N' \arrow{r} & N' \oplus_N T_{n-1} \arrow{r} & T_{n-2} \arrow{r} &  \cdots \arrow{r} & T_0 \arrow{r} & M \arrow{r} & 0
			
			\arrow[from=2-3, to=1-2, "\blacksquare"{anchor=center, pos=0.125, rotate=180}, draw=none]
		\end{tikzcd}
	\end{center}
	and
	\begin{center}
		\begin{tikzcd}
			0 \arrow{r} & N \arrow{r} \arrow[equal]{d} & T_{n-1} \arrow{r} \arrow[equal]{d} & \cdots \arrow{r} & T_1 \arrow{r} \arrow[equal]{d} & T_0 \times_M M' \arrow{r} \arrow{d} & M' \arrow{r} \arrow{d}{\beta} & 0
			\\ 0 \arrow{r} & N \arrow{r} & T_{n-1} \arrow{r} & \cdots \arrow{r} & T_1 \arrow{r} & T_0 \arrow{r} & M \arrow{r} & 0
			
			\arrow[from=1-6, to=2-7, "\blacksquare"{anchor=center, pos=0.125, rotate=90}, draw=none]
		\end{tikzcd}
	\end{center}
	
	The bottom and upper row, respectively, of these diagrams are $n$-extensions that will be denoted by $\alpha_{\ast}(\xi)$ and $\beta^{\ast}(\xi)$. Clearly, both $\alpha_{\ast}$ and $\beta^{\ast}$ preserve the equivalence relation of Definition \ref{defExModulos}.
\end{defn}

\begin{prop}
	Let $B$ be a ring and $M$ a $B$-module. With the above definition of $\alpha_{\ast}$, $\Ex_B^n(M, -)$ is a functor verifying the conditions stated in Proposition \ref{Prop1.1} and thus, if $B \to C$ is a ring homomorphism, $\Ex_B^n(M, N)$ is a $C$-module for any \break $C$-module $N$.
\end{prop}
\begin{proof}
	Given the $n$-extensions
	\[ \xi_1 = \left( 0 \to N_1 \to T_{n-1}^1 \to \cdots \to T_0^1 \to M \to 0 \right)\]
	\[ \xi_2 = \left( 0 \to N_2 \to T_{n-1}^2 \to \cdots \to T_0^2 \to M \to 0 \right),\]
	set $ \tau \colon \Ex_B^n(M, N_1) \times \Ex_B^n(M, N_2) \to \Ex_B^n(M, N_1 \oplus N_2)$ by
	\[ \tau([\xi_1],[\xi_2]) = \Delta^{\ast}(\xi_1 \oplus \xi_2)\]
	where
	\[ \xi_1 \oplus \xi_2 = \left( 0 \to N_1 \oplus N_2 \to T_{n-1}^1 \oplus T_{n-1}^2 \to \cdots \to T_0^1 \oplus T_0^2 \to M \oplus M \to 0 \right)\]
	and $\Delta \colon M \to M \oplus M$ is the diagonal homomorphism.
	
	One checks easily that $\Ex_B^n(M, 0)$ has a single element, which can be represented, for instance, by the extension
	\[ \left( 0 \to 0 \to 0 \to \cdots \to 0 \to M \xrightarrow{id} M \to 0 \right).\]
\end{proof}

\begin{rem}
	From the previous definitions, it is easy to decode what the operations are in $\Ex_B^n(M,N)$. For instance, the sum of the extensions
	\[ \xi = \left( 0 \to N \to T_{n-1} \to \cdots \to T_0 \to M \to 0 \right)\]
	\[ \xi' = \left( 0 \to N \to T'_{n-1} \to \cdots \to T'_0 \to M \to 0 \right)\]
	is given by the composition map $\Ex_B^n(M, N) \times \Ex_B^n(M, N) \xrightarrow{\tau} \Ex_B^n(M, N \oplus N) \xrightarrow{(+)_{\ast}} \Ex_B^n(M,N)$ and so $[\xi] + [\xi'] = \left[ (+)_{\ast} (\Delta^{\ast}(\xi \oplus \xi'))\right]$, where $+\colon N \oplus N \to N$, $(n,n') \mapsto n+n'$, and $\Delta \colon M \to M \oplus M$, $m \mapsto (m,m)$. This description is called the Baer sum, defined for instance in \cite[$\S$7, n.5, Remarque 2]{BoA10}. We also have $\left[ (+)_{\ast} (\Delta^{\ast}(\xi \oplus \xi'))\right]= \left[ \Delta^{\ast}(+)_{\ast} (\xi \oplus \xi')\right]$ by \cite[$\S$7, n.1 Exemple 3, n.4 Proposition 3]{BoA10}.
	
	If $b \in B$ and $(b \cdot) \colon N \to N$ is multiplication by $b$, then
	\[ b [\xi] = [(b \cdot)_{\ast}(\xi)]\]
	
	If $\alpha \colon N \to N'$ and $\beta \colon M' \to M$ are $B$-module homomorphisms, then so are $\alpha_{\ast} \colon \Ex_B^{n}(M,N) \to \Ex_B^{n}(M,N')$ and $\beta^{\ast} \colon \Ex_B^{n}(M,N) \to \Ex_B^{n}(M',N)$. A reference for this is \cite[\S7, n.4, Corollaire 3]{BoA10}, but a direct proof of this fact is easy, taking into account that the compositions $N \oplus N \xrightarrow{+} N \xrightarrow{\alpha} N'$ and $ N \oplus N \xrightarrow{(\alpha, \alpha)} N' \oplus N' \xrightarrow{+} N'$ agree (for the case $n=1$, also use the commutativity (naturality) of pullbacks and pushouts \cite[$\S$7, n.1 Exemple 3, n.4 Proposition 3]{BoA10} or \cite[Ch. 3, \S1, Lemma 1.6]{MacLane}).
%	\begin{center}
%		\begin{tikzcd}
%			% https://q.uiver.app/#q=WzAsMTIsWzAsMCwiTiBcXG9wbHVzIE4iXSxbMywwLCJOIFxcb3BsdXMgTiJdLFsxLDAsIlRfe24tMX0gXFxvcGx1cyBUJ197bi0xfSJdLFs0LDAsIlRfe24tMX0gXFxvcGx1cyBUJ197bi0xfSJdLFswLDEsIk4iXSxbMCwyLCJcXHRpbGRle059Il0sWzEsMSwiXFxidWxsZXQiXSxbMSwyLCJcXGJ1bGxldCJdLFszLDEsIlxcdGlsZGV7Tn0gXFxvcGx1cyBcXHRpbGRle059Il0sWzMsMiwiXFx0aWxkZXtOfSJdLFs0LDEsIlxcYnVsbGV0Il0sWzQsMiwiXFxidWxsZXQiXSxbMCw0XSxbNCw1XSxbMiw2XSxbNCw2XSxbNSw3XSxbNiw3XSxbMCwyXSxbMywxMF0sWzEwLDExXSxbOSwxMV0sWzgsMTBdLFsxLDhdLFs4LDldLFsxLDNdXQ==
%				{N \oplus N} & {T_{n-1} \oplus T'_{n-1}} && {N \oplus N} & {T_{n-1} \oplus T'_{n-1}} \\
%				N & \bullet && {\tilde{N} \oplus \tilde{N}} & \bullet \\
%				{\tilde{N}} & \bullet && {\tilde{N}} & \bullet
%				\arrow[from=1-1, to=1-2]
%				\arrow[from=1-1, to=2-1]
%				\arrow[from=1-2, to=2-2]
%				\arrow[from=1-4, to=1-5]
%				\arrow[from=1-4, to=2-4]
%				\arrow[from=1-5, to=2-5]
%				\arrow[from=2-1, to=2-2]
%				\arrow[from=2-1, to=3-1]
%				\arrow[from=2-2, to=3-2]
%				\arrow[from=2-4, to=2-5]
%				\arrow[from=2-4, to=3-4]
%				\arrow[from=2-5, to=3-5]
%				\arrow[from=3-1, to=3-2]
%				\arrow[from=3-4, to=3-5]
%		\end{tikzcd}
%	\end{center}	
\end{rem}

\begin{defn}
	Let $B$ be a ring, $M$, $N$ and $W$ $B$-modules, and
	\[ \xi_1 = \left( 0 \to N \to T_{n-1} \to \cdots \to T_0 \xrightarrow{\varepsilon} W \to 0 \right)\]
	\[ \xi_2 = \left( 0 \to W \xrightarrow{i} U_{m-1} \to \cdots \to U_0 \to M \to 0 \right)\]
	extensions. The Yoneda product is defined as
	\[ \xi_1 \circ \xi_2 = \left( 0 \to N \to T_{n-1} \to \cdots \to T_0 \xrightarrow{i  \varepsilon}  U_{m-1} \to \cdots \to U_0 \to M \to 0 \right)\]
	
	One checks immediately that this operation preserves the equivalence relation, and therefore defines a map
	\[ \Ex_B^n(W,N) \times \Ex_B^m(M,W) \to \Ex_B^{n+m}(M,N)\]
\end{defn}

\begin{prop}\label{Extex}
	There exists a natural $B$-module isomorphism
	\[ \Ext_B^n(M,N) \cong \Ex_B^n(M,N)\]
	taking the product defined in \cite[$\S$7, n.1]{BoA10} to the Yoneda product.
\end{prop}
\begin{proof}
	\cite[$\S$7]{BoA10}
	We will only give the idea of the proof, since we will use it frequently in what follows.
	
	Let $[f] \in \Ext_B^n(M,N)=\HH^n(\HHom_B(P_{\ast},N))$, where $P_{\ast}$ is a projective resolution of the $B$-module $M$ and $f \in \HHom_B(P_n, N)$ a cocycle, that is, the composition $P_{n+1} \xrightarrow{d_{n+1}} P_n \xrightarrow{f} N$ vanishes. As such, $f$ induces a map $\bar{f} \colon \ker(P_{n-1} \xrightarrow{d_{n-1}} P_{n-2}) \to N$. One assigns to $[f]$ the extension
	\[ (-1)^{n(n+1)/2} \bar{f}_{\ast}  \left( 0 \to \ker(P_{n-1} \xrightarrow{d_{n-1}} P_{n-2}) \to P_{n-1} \to \cdots \to P_0 \to M \to 0 \right)\]
	
	Conversely, for an extension
	\[ \xi = \left( 0 \to N \to T_{n-1} \to \cdots \to T_0 \to M \to 0 \right)\]
	one takes a projective resolution $P_{\ast}$ of the $B$-module $M$, and in the induced commutative diagram
	\begin{center}
		\begin{tikzcd}
			\cdots \arrow{r} & P_{n+1} \arrow{r} \arrow{d} & P_n \arrow{r} \arrow{d}{g} & P_{n-1} \arrow{r} \arrow{d} & \cdots \arrow{r} & P_0 \arrow{r} \arrow{d} & M \arrow{r} \arrow[equal]{d} & 0
			\\ \phantom{} & 0 \arrow{r} & N \arrow{r} & T_{n-1} \arrow{r} & \cdots \arrow{r} & T_0 \arrow{r} & M \arrow{r} & 0
		\end{tikzcd}
	\end{center}
	$(-1)^{n(n+1)/2} g \in \HHom_B(P_n, N)$ is an $n$-cocyle representing an element $[g]$ in \break $\Ext_B^n(M,N)$.
\end{proof}

\begin{rem} \label{freeExtension}
	From the preceding construction one deduces that every element of $\Ex_B^n(M,N)$ admits a representant of the form
	\[ \left( 0 \to N \to T_{n-1} \to P_{n-2} \cdots \to P_0 \to M \to 0 \right)\]
	where $P_0, \dots, P_{n-2}$ are free $B$-modules.
\end{rem}

\begin{defn}
	Let $A \to B$ be a ring homomorphism and $M$ a $B$-module. An \emph{infinitesimal extension} of $B$ over $A$ by the $B$-module $M$ is an exact sequence of $A$-modules
	\[ 0 \to M \xrightarrow{i} E \xrightarrow{p} B \to 0\]
	where $E$ is an $A$-algebra, $p$ an $A$-algebra homomorphism and $e \cdot i(m) = i(p(e) \cdot m)$ for any $e \in E$, $m \in M$. Note that, therefore, $i(m_1) \cdot i(m_2) = i(p(i(m_1)) \cdot m_2) = i(0 \cdot m_2) = 0$, that is, $i(M)^2=0$.
\end{defn}

Reciprocally:

\begin{rem}
	Let $A$ be a ring, $p \colon E \to B$ a surjective $A$-algebra homomorphism with kernel the square zero ideal $M$. Then
	\[ 0 \to M \xhookrightarrow{} E \xrightarrow{p} B \to 0\]
	is an infinitesimal extension of $B$ by $M$. Indeed, $M = M/M^2$ is a $B$-module and writing $i$ for the inclusion of $M$ in $E$, clearly $e \cdot i(m) = i(p(e) \cdot m)$.
\end{rem}

\begin{defn}
	Two infinitesimal extensions of $B$ over $A$ by $M$
	\[ 0 \to M \xrightarrow{i} E \xrightarrow{p} B \to 0\]
	\[ 0 \to M \xrightarrow{i'} E' \xrightarrow{p'} B \to 0\]
	are said to be \emph{equivalent} if there exists a homomorphism of $A$-algebras $\alpha: E \to E'$ making commutative the diagram
	\begin{center}
		\begin{tikzcd}
			0 \arrow{r} & M \arrow{r}{i} \arrow[equal]{d} & E \arrow{r}{p} \arrow{d}{\alpha} & B \arrow{r} \arrow[equal]{d} & 0
			\\ 0 \arrow{r} & M \arrow{r}{i'} & E' \arrow{r}{p'} & B \arrow{r} & 0
		\end{tikzcd}
	\end{center}

	If this is the case, then $\alpha$ is an isomorphism of $A$-algebras, and thus this defines an equivalence relation. The set of equivalence classes will be denoted by $\Exal_A(B,M)$.
\end{defn}

\begin{prop}
	$\Exal_A(B,M)$ is a contravariant functor in $A$ and in $B$. It is also a covariant functor in $M$ which satisfies the conditions from Proposition \ref{Prop1.1}.
\end{prop}
\begin{proof}
	In $A$ it simply acts as a forgetful functor.
	
	If $f \colon B' \to B$ is an $A$-algebra homomorphism and $\xi = ( 0 \to M \xrightarrow{i} E \xrightarrow{p} B \to 0)$ is an infinitesimal extension, then it is easy to check that
	\[  f^{\ast}(\xi) = \left(0 \to M \to E \times_B B' \to B' \to 0\right) \]
	is an infinitesimal extension.
	
	If $g \colon M \to M'$ is a $B$-module homomorphism, then
	\[ g_{\ast}(\xi) = \left( 0 \to M' \to M' \oplus_M E \to B \to 0 \right)\]
	is also an infinitesimal extension (in which the $A$-algebra structure for $M' \oplus_M E$ is given by $ (m'_1,e_1) \cdot (m'_2,e_2) = (e_1 m'_2 + e_2 m'_1,e_1 e_2)$, considering $M$ as an $E$-module via $p$).
	
	If
	\[ \xi_1 = \left(0 \to M_1 \to E_1 \to B \to 0\right)\]
	\[ \xi_2 = \left(0 \to M_2 \to E_2 \to B \to 0\right)\]
	are two infinitesimal extensions of $B$ over $A$, we define the homomorphism $\tau$ of Proposition \ref{Prop1.1} as
	\[ \tau([\xi_1], [\xi_2]) = \Delta^{\ast}(0 \to M_1 \oplus M_2 \to E_1 \times E_2 \to B \times B \to 0)\]
	where $\Delta \colon B \to B \times B$ is the diagonal homomorphism.
	
	Details are easy to check.
\end{proof}

\begin{cor}
	$\Exal_A(B,M)$ has a $B$-module structure. \qed
\end{cor}

One can easily check that the equivalence class of $0$ is the set of extensions $0 \to M \to E \xrightarrow{p} B \to 0$ admitting an $A$-algebra section $s \colon B \to E$.

\begin{prop} \label{H1Exalcom}
	There is a $B$-module isomorphism, natural in $A$, $B$ and $M$,
	\[ \HH^1(A,B,M) \cong \Exal_A(B,M) \]
	where $\HH^1(A,B,M)$ is the first Andr\'e-Quillen cohomology module.
\end{prop}
\begin{proof}
	 \cite[16.12]{An1974}.
\end{proof}

\section{Mehta's construction}

In this section, when we speak of the connecting homomorphism for the $\Tor$ long exact sequence, we will refer to the canonical choice of this homomorphism, as done in \cite[$\S$4, n.5]{BoA10}.

\begin{defn} \label{homMehta}
	Let $R \to B \to C$ be surjective ring homomorphisms and $N$ a $C$-module. Define a $C$-module homomorphism
	\[ \varphi \colon \Ex_R^1(B,C) \to \Ex_B^2(N,N) \]
	in the following way: for an exact sequence of $R$-modules $0 \to K \to F \to N \to 0 $ with $F$ projective, take the exact sequence
	\[ \theta = \left(0 \to \Tor_1^R(B,N) \to B \otimes_R K \to B \otimes_R F \to B \otimes_R N = N \to 0 \right) \]
	
	Let $[\xi] = \left[ 0 \to C \to T \to B \to 0 \right] \in \Ex_R^1(B,C)$, and let
	\[ \partial \colon \Tor_1^R(B,N) \to \Tor_0^R(C,N) = N \]
	be the connecting homomorphism associated to $\xi$ and $- \otimes_R N$.
	
	Set $\varphi[\xi] \eqdef [\partial_{\ast}(\theta)] \in \Ex_B^2(N,N)$.
\end{defn}

This definition was given by Mehta (for the main case, $C=B$) \cite{Mehta} and, in the case in which $I = \ker(R \to B)$ is generated by a regular sequence $x_1, \dots, x_n$, if $\hat{x_i} \in \HHom_B(I/I^2,B) = \Ex
_R^1(B,B)$ denotes the homomorphism given by $\hat{x_i}(x_i)=1$, $\hat{x_i}(x_j)=0$ for $i \neq j$, Mehta proved that the operation of this $\hat{x_i}$ onto $\Ex_B^{\ast}(N,M)$ in the sense of Gulliksen \cite{Gulliksen} is precisely the Yoneda product by the extension $\varphi(\hat{x_i})$.

\begin{prop} \label{MehtaProof}
	$\varphi$ is well-defined (it depends only on the class of the extension $\xi$, and not on the chosen projective presentation $F \to N$) and a $C$-module homomorphism.
\end{prop}
\begin{proof}
	It is easy to see that it is well-defined. We will show that it is a $C$-module homomorphism, since the proof does not appear in \cite{Mehta}.
	
	Let
	\[ [\xi_j] = \left[0 \to C \xrightarrow{i_j} T_j \xrightarrow{p_j} B \to 0\right] \in \Ex_R^1(C,B), \mkern9mu j = 1,2\]
	and fix an exact sequence of $R$-modules $0 \to K \to F \to N \to 0$ with $F$ projective. As in Definition \ref{homMehta}, consider the induced exact sequence
	\[ \theta = \left(0 \to \Tor_1^R(B,N) \xrightarrow{\partial_N} B \otimes_R K \to B \otimes_R F \to B \otimes_R N = N \to 0 \right), \]
	so one has $\varphi[\xi_j] = [(\partial_j)_{\ast}(\theta)]$, $\partial_j \colon \Tor_1^R(B,N) \to \Tor_0^R(C,N) = N$ being the connecting homomorphisms associated with $\xi_j$.
	
	Now, consider the Baer sum $[\xi_1] + [\xi_2]$, given by the class of the bottom row of the diagram
	\begin{center}
		\begin{tikzcd}
			0 \arrow{r} & C \oplus C \arrow{r} & T_1 \oplus T_2 \arrow{r} & B \oplus B \arrow{r} & 0
			\\ 0 \arrow{r} & C \oplus C \arrow{r} \arrow[equal]{u} \arrow{d}{+} & T_1 \times_B T_2 \arrow{r} \arrow{u} \arrow{d} & B \arrow{r} \arrow{u}{\Delta} \arrow[equal]{d} & 0
			\\ 0 \arrow{r} & C \arrow{r} & W \arrow{r} & B \arrow{r} & 0
			
			\arrow[from=2-3, to=1-4, "\blacksquare"{anchor=center, pos=0.125, rotate=90}, draw=none]
			\arrow[from=3-3, to=2-2, "\blacksquare"{anchor=center, pos=0.125, rotate=180}, draw=none]
		\end{tikzcd}
	\end{center}
	and note that this diagram induces the following one, for the long exact homology sequence
	\begin{center}
		\begin{tikzcd}[scale cd=0.68, sep=small]
			\Tor_1^R(B, N) \oplus \Tor_1^R(B, N) \arrow{r}{(\partial_1, \partial_2)} & (C \otimes_R N) \oplus (C \otimes_R N) \arrow{r} & (T_1 \otimes_R N) \oplus (T_2 \otimes_R N) \arrow{r} & (B \otimes_R N) \oplus (B \otimes_R N) \arrow{r} & 0
			\\ \Tor_1^R(B \oplus B, N) \arrow{r} \arrow[equal]{u} & (C \oplus C) \otimes_R N \arrow{r} \arrow[equal]{u} & (T_1 \oplus T_2) \otimes_R N \arrow{r} \arrow[equal]{u} & (B \oplus B) \otimes_R N \arrow{r} \arrow[equal]{u} & 0
			\\ \Tor_1^R(B,N) \arrow{r} \arrow{u}{\Delta_1} \arrow[equal]{d} & (C \oplus C) \otimes_R N \arrow{r} \arrow[equal]{u} \arrow{d}{+} & (T_1 \times_B T_2) \otimes_R N \arrow{r} \arrow{u} \arrow{d} & B \otimes_R N \arrow{r} \arrow{u}{\Delta_0} \arrow[equal]{d} & 0
			\\ \Tor_1^R(B,N) \arrow{r}{\partial_W} & C \otimes_R N \arrow{r} & W \otimes_R N \arrow{r} & B \otimes_R N \arrow{r} & 0
		\end{tikzcd}
	\end{center}
	in which $\partial_W$ denotes the connecting homomorphism used to compute $\varphi([\xi_1] + [\xi_2])$. The diagram shows that $\partial_W = + \circ (\partial_1, \partial_2) \circ \Delta_1$.
	
	The explicit computations of $\varphi[\xi_j]$ for $j=1,2$ are given by the diagrams
	\begin{center}
		\begin{tikzcd}
			0 \arrow{r} & \Tor_1^R(B,N) \arrow{r}{\partial_N} \arrow{d}{\partial_j} & B \otimes_R K \arrow{r} \arrow{d}{\beta_j} & B \otimes_R F \arrow{r} \arrow[equal]{d} & B \otimes_R N = N \arrow{r} \arrow[equal]{d} & 0
			\\ 0 \arrow{r} & N \arrow{r}{\alpha_j} & P_j \arrow{r} & B \otimes_R F \arrow{r} & N \arrow{r} & 0
			
			\arrow[from=2-3, to=1-2, "\blacksquare"{anchor=center, pos=0.125, rotate=180}, draw=none]
		\end{tikzcd}
	\end{center}
	so $\varphi[\xi_1] + \varphi[\xi_2]$ is just the equivalence class of the bottom row of the following diagram
	\begin{center}
		\begin{tikzcd}[column sep=0.8cm]
			0 \arrow{r} & N \oplus N \arrow{r}{(\alpha_1, \alpha_2)} & P_1 \oplus P_2 \arrow{r} & (B \otimes_R F) \oplus (B \otimes_R F) \arrow{r} & N \oplus N \arrow{r} & 0
			\\ 0 \arrow{r} & N \oplus N \arrow{r}{(\alpha_1, \alpha_2)} \arrow[equal]{u} \arrow{d}{+} & P_1 \oplus P_2 \arrow{r} \arrow[equal]{u} \arrow{d}{\pi} & (B \otimes_R F) \times_N (B \otimes_R F) \arrow{r} \arrow{u} \arrow[equal]{d} & N \arrow{r} \arrow{u}{\Delta} \arrow[equal]{d} & 0
			\\0 \arrow{r} & N \arrow{r}{h} & P_{1,2} \arrow{r} & (B \otimes_R F) \times_N (B \otimes_R F) \arrow{r} & N \arrow{r} & 0
			
			\arrow[from=2-4, to=1-5, "\blacksquare"{anchor=center, pos=0.125, rotate=90}, draw=none]
			\arrow[from=3-3, to=2-2, "\blacksquare"{anchor=center, pos=0.125, rotate=180}, draw=none]
		\end{tikzcd}
	\end{center}
	
	To compare this extension with 
	\[\varphi([\xi_1]+[\xi_2]) = [(\partial_W)_{\ast}(\theta)] = \left[ 0 \to N \to P_W \to B \otimes_R F \to N \to 0\right],\]
	one can induce a map $\eta \colon P_W \to P_{1,2}$ by the universal property
	\begin{center}
		\begin{tikzcd}
			\Tor_1^R(B,N) \arrow{r}{\partial_N} \arrow{d}{\partial_W} & B \otimes_R K \arrow{d} \arrow[bend left]{ddr}{\pi \circ (\beta_1, \beta_2) \circ \Delta} &
			\\ N \arrow{r} \arrow[bend right,swap]{drr}{h} & P_W \arrow[dashed]{dr}{\eta} &
			\\ & & P_{1,2}
			
			\arrow[from=2-2, to=1-1, "\blacksquare"{anchor=center, pos=0.125, rotate=180}, draw=none]
		\end{tikzcd}
	\end{center}
	which can be used to show that $\varphi([\xi_1]+[\xi_2])$ and $\varphi[\xi_1] + \varphi[\xi_2]$ are directly related:
	\begin{center}
		\begin{tikzcd}
			0 \arrow{r} & N \arrow{r} \arrow[equal]{d} & P_W \arrow{r} \arrow{d}{\eta} & B \otimes_R F \arrow{r} \arrow{d}{\Delta} & N \arrow{r} \arrow[equal]{d} & 0
			\\0 \arrow{r} & N \arrow{r}{h} & P_{1,2} \arrow{r} & (B \otimes_R F) \times_N (B \otimes_R F) \arrow{r} & N \arrow{r} & 0
		\end{tikzcd}
	\end{center}
	
	Checking that $\varphi[c \cdot \xi] = c \cdot \varphi[\xi]$ for all $c \in C$ and $[\xi] \in \Ex_R^1(B,C)$ is straightforward.
\end{proof}

\section{The homomorphism $\psi$}

\begin{defn} \label{homPsi}
	Let $A \to B \xrightarrow{q} C$ be ring homomorphisms, with $q$ surjective, and $N$ a $C$-module. Define a $C$-module homomorphism
	\[ \psi \colon \Exal_A(B,C) \to \Ex_B^2(N,N) \]
	in the following way: for an infinitesimal extension
	\[ [\xi] = \left[0 \to C \to E \xrightarrow{\pi} B \to 0\right] \in  \Exal_A(B,C)\]
	and an exact sequence of $E$-modules $0 \to K \to F \to N \to 0 $ with $F$ projective ($N$ is an $E$-module via $\pi$), take the exact sequence
	\[ \vartheta = \left(0 \to \Tor_1^E(B,N) \to B \otimes_E K \to B \otimes_E F \to B \otimes_E N = N \to 0 \right) \]
	
	The extension $\xi$ induces an exact sequence
	\[ 0 \to \Tor_1^E(B,N) \xrightarrow{\partial} C \otimes_E N \to E \otimes_E N \xrightarrow{\cong} B \otimes_E N  \to 0 \]
	Identifying $\Tor_1^E(B,N)$ with $C \otimes_E N\cong N$ through the canonical homomorphism $\partial$ in the exact sequence $\vartheta$, we obtain $\psi([\xi]) \in \Ex_B^2(N,N)$, that is,
	\[ \psi([\xi]) \eqdef [\partial_{\ast}(\vartheta)]. \]
	
\end{defn}

It can be easily checked that $\psi$ is well-defined. We will see in Proposition \ref{ProofPsi} that it is a $C$-module homomorphism.

\begin{rem} \label{homAlpha}
	Let $A \to R \xrightarrow{\pi} B \xrightarrow{q} C$ be ring homomorphisms, with $\pi$ and $q$ surjective, and $M$ a $C$-module. Let $I = \ker(\pi)$. We have a natural homomorphism
	\[ \alpha \colon \Ex_R^1(B,M) \to \Exal_A(B,M),\]
	defined as follows. Let
	\[ [\xi] = \left[ 0 \to M \to T \to B \to 0 \right] \in \Ex_R^1(B,M) \]
	Since $R$ is a free $R$-module, we have a commutative diagram of vertical $R$-linear maps
	\begin{center}
		\begin{tikzcd}
			0 \arrow{r} & I \arrow{r} \arrow[dashed]{d}{g} & R \arrow{r} \arrow{d}{f} &  B \arrow{r} \arrow[equal]{d} & 0
			\\0 \arrow{r} & M \arrow{r} & T \arrow{r} &  B \arrow{r} & 0
		\end{tikzcd}
	\end{center}
	
	We define $\alpha([\xi]) \eqdef \left[g_{\ast}(0 \to I \to R \to B \to 0) \right] = \left[0 \to M \to M \oplus_I R \to B \to 0\right].$
\end{rem}

As before, a routine check shows that $\alpha$ is well-defined and it is a $C$-module homomorphism. In fact, identifying $\Ex_R^1(B,M)$ with $\Exal_R^1(B,M)$ by Proposition \ref{Extex} and \cite[6.1]{An1974}, the homomorphism $\alpha$ (thought at the level of Andr\'e-Quillen homology, see Proposition \ref{H1Exalcom}) is the homomorphism in the Jacobi-Zariski exact sequence \cite[5.1]{An1974}.

It is easy to see that $\psi$ is natural in the following sense:
\begin{defn}
	Let $A \to B \xrightarrow{q} C$ be ring homomorphisms, with $q$ surjective, and $N$ a $C$-module. Define
	\[ \psi_{A,B,C,N} \colon \Exal_A(B,C) \to \Ex_B^2(N,N)\]
	as in Definition \ref{homPsi}. If $q = \mathrm{id}_B$, we abridge it to $\psi_{A,B,N}$.
\end{defn}

\begin{prop}
	\begin{enumerate}
		\item[(i)] If $A \to R \xrightarrow{p} B \xrightarrow{q} C $ are ring homomorphisms with $q$ surjective and $N$ is a $C$-module, the following diagram commutes:
		\begin{center}
			\begin{tikzcd}
				\Exal_R(B,C) \arrow{r} \arrow{d}{\psi_{R,B,C,N}} & \Exal_A(B,C) \arrow{d}{\psi_{A,B,C,N}}
				\\ \Ex_B^2(N,N) \arrow[equal]{r} & \Ex_B^2(N,N)
			\end{tikzcd}
		\end{center}
		\item[(ii)] If $qp$ is surjective, the following diagram is also commutative:
		\begin{center}
			\begin{tikzcd}
				\Exal_A(B,C) \arrow{r} \arrow{d}{\psi_{A,B,C,N}} & \Exal_A(R,C) \arrow{d}{\psi_{A,R,C,N}}
				\\ \Ex_B^2(N,N) \arrow{r} & \Ex_R^2(N,N)
			\end{tikzcd}
		\end{center}
		\item[(iii)] If $A \to B \xrightarrow{q} C \xrightarrow{q'} C' $ are ring homomorphisms with $q$ and $q'$ surjective and $N$ is a $C$-module, we have a commutative diagram:
		\begin{center}
			\begin{tikzcd}
				\Exal_A(B,C) \arrow{r} \arrow{d}{\psi_{A,B,C,N}} & \Exal_A(B,C') \arrow{d}{\psi_{A,B,C',N}}
				\\ \Ex_B^2(N,N) \arrow[equal]{r} & \Ex_B^2(N,N)
			\end{tikzcd}
		\end{center}
	\end{enumerate}
\end{prop}

\begin{thm} \label{commPhiPsi}
	Let $A \to R \xrightarrow{\pi} B \xrightarrow{q} C $ be ring homomorphisms, with $\pi$ and $q$ surjective, and $N$ a $C$-module. The triangle
	\begin{center}
		\begin{tikzcd}
			\Ex_R^1(B,C) \arrow{rr}{\alpha} \arrow[swap]{dr}{\varphi} & & \Exal_A(B,C) \arrow{dl}{\psi}
			\\ & \Ex_B^2(N,N) &
		\end{tikzcd}
	\end{center}
	is commutative.
\end{thm}
\begin{proof}
	We have seen in Remark \ref{homAlpha} that we have an isomorphism $\Ex_R^1(B,C) \to \Exal_R(B,C)$ whose inverse is just forgetting the additional structure of the extension. Hence, one can take any element from $\Ex_R^1(B,C)$ to be the class (in $\Ex_R^1(B,C)$) of an infinitesimal extension over $R$
	\[ \xi = \left(0 \to C \to E \to B \to 0\right)\]
	Thus, $\alpha([\xi])$ will be the equivalence class of $\xi$, thought of as an extension over $A$.
	
	Let $0 \to K \to F \to N \to 0 $ be an exact sequence of $E$-modules, with $F$ projective. Then $\psi \alpha ([\xi])$ is the equivalence class in $\Ex_B^2(N,N)$ of the bottom row of the diagram
	\begin{center}
		\begin{tikzcd}[column sep=0.7cm]
			0 \arrow{r} & \Tor_1^E(B,N) \arrow{r} \arrow[d, "\partial", "\cong"'] & B \otimes_E K \arrow{r} \arrow{d}{\cong} &  B \otimes_E F \arrow{r} \arrow[equal]{d} &  B \otimes_E N = N \arrow{r} \arrow[equal]{d} & 0
			\\ 0 \arrow{r} & \Tor_0^E(C,N)=N \arrow{r} & T \arrow{r} &  B \otimes_E F \arrow{r} &  N \arrow{r} & 0
			
			\arrow[from=2-3, to=1-2, "\blacksquare"{anchor=center, pos=0.125, rotate=180}, draw=none]
		\end{tikzcd}
	\end{center}
	where $\partial$ is the connecting homomorphism associated with the exact sequence \break $0 \to C \to E \to B \to 0$.
	
	Now consider an exact sequence of $R$-modules
	\[ 0 \to K' \to F' \to N \to 0\]
	with $F'$ projective. Then $\varphi([\xi])$ is the equivalence class of the bottom row of the diagram
	\begin{center}
		\begin{tikzcd}[column sep=0.65cm]
			0 \arrow{r} & \Tor_1^R(B,N) \arrow{r} \arrow{d}{\tilde{\partial}} & B \otimes_R K' \arrow{r} \arrow{d} &  B \otimes_R F' \arrow{r} \arrow[equal]{d} &  B \otimes_R N = N \arrow{r} \arrow[equal]{d} & 0
			\\ 0 \arrow{r} & \Tor_0^R(C,N)=N \arrow{r} & W \arrow{r} &  B \otimes_R F' \arrow{r} &  N \arrow{r} & 0
			
			\arrow[from=2-3, to=1-2, "\blacksquare"{anchor=center, pos=0.125, rotate=180}, draw=none]
		\end{tikzcd}
	\end{center}
	where $\tilde{\partial}$ is the connecting homomorphism associated with $0 \to C \to E \to B \to 0$, but this time over the ring $R$.
	
	The commutative diagram of $R$-module maps (recall that $E$ is an $R$-algebra)
	\begin{center}
		\begin{tikzcd}
			0 \arrow{r} & K' \arrow{r} \arrow[dashed]{d} & F' \arrow{r} \arrow[dashed]{d} &  N \arrow{r} \arrow[equal]{d} & 0
			\\0 \arrow{r} & K \arrow{r} & F \arrow{r} &  N \arrow{r} & 0
		\end{tikzcd}
	\end{center}
	induces the following commutative diagram
	\begin{center}
		\begin{tikzcd}[scale cd= 0.68, column sep= 0.1 cm]
			0 \arrow{rr} & & \Tor_1^R(B,N) \arrow{rr} \arrow{dr} \arrow{dd}{\tilde{\partial}} & & B \otimes_R K' \arrow{rr} \arrow{dd} \arrow{dr} & & B \otimes_R F' \arrow{rr} \arrow[equal]{dd} \arrow{dr} & & B \otimes_R N = N \arrow{rr} \arrow[equal]{dd} \arrow[equal]{dr} &  & 0 &
			
			\\ &  0 \arrow[crossing over]{rr} & & \Tor_1^E(B,N) \arrow[crossing over]{rr}  & & B \otimes_E K \arrow[crossing over]{rr}  & & B \otimes_E F \arrow[crossing over]{rr}  & & B \otimes_E N = N \arrow{rr} \arrow[equal]{dd} & & 0
			
			\\ 0 \arrow{rr} & & N \arrow{rr} \arrow[equal]{dr} & & W \arrow{rr} \arrow[dashed]{dr} & &  B \otimes_R F' \arrow{rr} \arrow{dr} & & N \arrow{rr} \arrow[equal]{dr} & & 0 &
			
			\\ & 0 \arrow{rr} & & N \arrow{rr} & & T \arrow{rr} & & B \otimes_E F \arrow{rr} & & N \arrow{rr} & & 0
			
			\arrow[from=2-4, to=4-4, crossing over]
			\arrow[from=2-6, to=4-6, crossing over]
			\arrow[from=2-8, to=4-8, crossing over, equal]
			\arrow[from=2-8, to=4-8, crossing over, equal]
		\end{tikzcd}
	\end{center}
	which shows that $\varphi([\xi])=\psi\alpha([\xi])$.
\end{proof}

\begin{prop} \label{ProofPsi}
	With the notation from Definition \ref{homPsi}, $\psi$ is a $C$-module homomorphism.
\end{prop}
\begin{proof}
	Let $R$ be a polynomial $A$-algebra such that there is an $A$-algebra surjection $R \to B$. We have a commutative diagram
	\begin{center}
		\begin{tikzcd}
			\Exal_R(B,C) \arrow{r}{\tilde{\alpha}} \arrow{d}{\psi_{R,B,C,N}} & \Exal_A(B,C) \arrow{r} \arrow{d}{\psi_{A,B,C,N}} & \Exal_A(R,C)
			\\ \Ex_B^2(N,N) \arrow[equal]{r}  & \Ex_B^2(N,N)
		\end{tikzcd}
	\end{center}
	in which the upper row is exact by Proposition \ref{H1Exalcom} and \cite[5.1]{An1974}, and moreover $\Exal_A(R,C)=0$ \cite[3.36]{An1974}.
	
	Therefore, $\tilde{\alpha}$ is surjective and so it suffices to show that $\psi_{R,B,C,N}$ is a homomorphism to see that $\psi_{A,B,C,N}$ is one as well. By Theorem \ref{commPhiPsi} and Remark \ref{homAlpha}, $\psi_{R,B,C,N}$ is precisely Mehta's $\varphi$, which was already proved to be a homomorphism in Proposition \ref{MehtaProof}.
\end{proof}

\begin{prop}\label{central}
	Let $A \to B \xrightarrow{q} C$ be ring homomorphisms, with $q$ surjective, $M$ a $C$-module, $[\xi], [\eta] \in \Exal_A(B,C)$. Then
	\[ \psi[\xi] \circ \psi[\eta] = \psi[\eta] \circ \psi[\xi]\]
	where $\circ$ denotes the Yoneda product.
	
	Therefore, for every $C$-module $N$, we have a $\Symm_C^{\ast}(\Exal_A(B,C))$-module structure on $\Ex_B^{\ast}(N,M)$ induced by $ [\xi] \cdot [\vartheta] \eqdef \psi[\xi] \circ [\vartheta]$ for $[\xi] \in \Exal_A(B,C)$ and $[\vartheta] \in \Ex_B^{\ast}(N,M)$.
\end{prop}
\begin{proof}
	In a similar vein to the proof of Proposition \ref{ProofPsi}, we can limit ourselves to proving the proposition for $\varphi$, instead of $\psi$. But this follows from \cite[Proposition 2.3]{Mehta} (see also \cite[Remark (1), p. 706]{AvramovSun}).
\end{proof}

\begin{rem} \label{h2}
	Let $B=R/I$. Choosing a ring $A$ such that $R$ is an $A$-algebra, Theorem \ref{commPhiPsi} allows us to factorize Gulliksen operations from $\Ext_R^1(B,B)= \widehat{I/I^2}$ through $\HH^1(A,B,B)$. If we choose  $A$ so that $\HH^1(A,R,B)=0$ (for instance if $R$ is noetherian and formally smooth for the $I$-adic topology \cite[16.17]{An1974}), then the Jacobi-Zariski exact sequence shows that $\alpha$ is surjective, so the action on $\Ext_B^{\ast}(N,N) $ is ``more faithful'' from the ring $\Symm_B^{\ast} \HH^1(A,B,B)$ than from $\Symm_B^{\ast}(\widehat{I/I^2})$. But in general we cannot always expect to find such a ring $A$ so that the action of $\Symm_B^{\ast}(\HH^1(A,B,B))$ on $\Ext_B^{\ast}(N,N) $ be faithful. Take $R= \mathbb{Z}_{(p)}$ the local ring at the prime $(p)$ and $B$ its residue field $F_p$ so that the cohomological operations on $\Ext_{F_p}^{\ast}$ are zero. We cannot find a ring homomorphism $f \colon A \to R$ such that $\HH^1(A,B,B)=0$: if we had $f \colon A \to R$, assuming $A$ local with maximal ideal $\qqq$ and $f$ local (localizing $A$ at $f^{-1}(p\mathbb{Z}_p)$, since $ \HH^1(A,B,B) =\HH^1(A_{\qqq},B,B) $), the ring map $\mathbb{Z} \to A$ induces $\mathbb{Z}_p \to A$ such that the composition  $\mathbb{Z}_p \to A \xrightarrow{f} \mathbb{Z}_p$ is the identity map. Therefore, $\HH^1(\mathbb{Z}_p, F_p,F_p) \neq 0$ is a direct summand of $\HH^1(A, F_p,F_p)$.
\end{rem}

Let $f \colon A \to B$ be a complete intersection local homomorphism in the sense of \cite{AvramovAnnals}, that is, $f$ is a local homomorphism of noetherian local rings such that $\HH_n(A,B, -)=0$ for all $n \geq 2$. Assume that there exists a regular factorization $A \xrightarrow{u} R \xrightarrow{v} B$ of $f$, meaning that $u$ is a flat local homomorphism of noetherian rings with regular closed fiber and $v$ is surjective (this happens for instance if $B$ is complete \cite{AvramovFoxbyHerzog} or if $f$ is essentially of finite type).

\begin{thm}[Finite generation] \label{finiteGeneration}
	Under these hypotheses, $\Ext_B^{\ast}(M,N)$ is a \break $\Symm_B^{\ast}(\HH^1(A,B,B))$-module of finite type for any pair of $B$-modules of finite type $M$ and $N$ such that $\pdim_A(M) < \infty$ or $\idim_A(N) < \infty$.	
\end{thm}
\begin{proof}
	By Theorem \ref{commPhiPsi}, it suffices to show that $\Ext_B^{\ast}(M,N)$ is a $\Symm_B^{\ast}(\Ext_R^1(B,B))$-module of finite type.
	
	Since $f$ is complete intersection, $\ker(v)$ is generated by a regular sequence. Let $k$ be the residue field of $A$. Flatness of the homomorphism $u$ means that we have isomorphisms $\Tor_q^R(R \otimes_A k, M) = \Tor_q^A(k,M)$, $\Ext_R^q(R \otimes_A k, N) = \Ext_A^q(k,N) $. Since $R \otimes_A k$ is regular, from the change of rings spectral sequences, where $l$ is the residue field of $B$,
	\begin{align*}
		E^2_{pq}= \Tor^{R \otimes_A k}_p(\Tor_q^R(R \otimes_A k, M),l) \Rightarrow \Tor^R_{p+q}(M,l)
		\\E_2^{pq}= \Ext_{R \otimes_A k}^p(l,\Ext_R^q(R \otimes_A k, N)) \Rightarrow \Ext_R^{p+q}(l,N)
	\end{align*}
	we deduce that $\Tor^R_n(M,l)=0$ or $\Ext^n_R(l,N)=0$ for all $n >>0$, respectively, if $\pdim_A(M) < \infty$ or $\idim_A(N) < \infty$, whence $\pdim_R(M) < \infty$ or $\idim_R(N) < \infty$. Then the result follows from \cite[Theorem 3.1]{Gulliksen}. 
\end{proof}

\section{Compatibility of $\psi$ with Hochschild cohomology} \label{Section4}
Let $k$ be a field, $B$ a commutative $k$-algebra and $\Hochs^{\ast}(B | k) = \Ext_{B \otimes_k B}^{\ast}(B,B)$ its Hochschild cohomology. For any $B$-module $N$, we have a homomorphism
\[ s \colon \Hochs^2(B | k) \to \Ext_B^2(N,N)\]
defined in the following way (see for instance \cite{SnaSol}).

One can think of an element of $\Hochs^2(B | k)$ as the equivalence class in $\Ex_{B \otimes_k B}^2(B,B)$ of an extension of $B \otimes_k B$-modules
\[ \vartheta = (0 \to B \to T \to P \to B \to 0)\]
where $P$ is a free $B \otimes_k B$-module (Remark \ref{freeExtension}). Considering this extension as a sequence of $B$-modules via the homomorphism $i_2 \colon B \to B \otimes_k B$, $i_2(b)= 1 \otimes b$, it becomes an exact sequence of flat $B$-modules and thus, applying $ - \otimes_B N$,
\[ 0 \to B \otimes_B N = N \to T \otimes_B N \to P \otimes_B N \to B \otimes_B N = N \to 0 \]
is an exact sequence. Define $s([\vartheta])$ as the equivalence class of this extension, considered as a $B$-module extension via the homomorphism $i_1 \colon B \to B \otimes_k B$, $i_1(b)= b \otimes 1$.

On the other hand, we have a canonical homomorphism $\gamma \colon \HH^1(k,B,B) \to \Hochs^2(B | k)$ (see definition below). We have
\begin{thm} \label{commSnashallSolberg}
	The triangle
	\begin{center}
		\begin{tikzcd}
			\HH^1(k,B,B) \arrow{rr}{\gamma} \arrow[swap]{dr}{\psi} & & \Hochs^2(B | k) \arrow{dl}{s}
			\\ & \Ex_B^2(N,N) & 
		\end{tikzcd}
	\end{center}
	commutes.
\end{thm}

Before proceeding with the proof, we briefly make the following considerations:

\subsection*{Bar resolution} \hfill \break
Let $k$ be a field, $B$ a commutative $k$-algebra. We have a complex of $B \otimes_k B$-modules denoted $\mathcal{B}_{\ast}(B)$ and defined as
\[ \mathcal{B}_n(B) = B \otimes_k \dots \otimes_k B \mkern9mu \text{(}n+2 \text{ factors)}\]
where the $B \otimes_k B$-module structure is given by
\[ (b \otimes b') (b_0 \otimes \dots \otimes b_{n+1}) = b b_0 \otimes b_1 \otimes \dots \otimes b_n \otimes b_{n+1} b' \]
with differential $d_n \colon \mathcal{B}_n(B) \to \mathcal{B}_{n-1}(B)$ defined by
\[ d_n(b_0 \otimes \dots \otimes b_{n+1}) = \sum_{i=0}^{n} (-1)^i b_0 \otimes \dots \otimes b_i b_{i+1} \otimes \dots \otimes b_{n+1}\]

The complex $\mathcal{B}_{\ast}(B)$ has an augmentation 
\[d_0 \colon B \otimes_k B \to B, \mkern9mu d_0(b \otimes b') = b b'\]
and $\mathcal{B}_{\ast}(B)$ is a free resolution of the $ B \otimes_k B$-module $B$ (considering $B$ as a $ B \otimes_k B$-module via $d_0$). In fact, this augmented complex is nullhomotopic via the homomorphism $\alpha \mapsto 1 \otimes \alpha$.

Considering this augmented complex
\[ \cdots \to B \otimes_k B \otimes_k B \to B \otimes_k B \to B \to 0\]
as a complex of $B$-modules by restriction of scalars via the $k$-algebra homomorphism $i_2 \colon B \to B \otimes_k B, \mkern9mu i_2(b) = 1 \otimes b$, it is an exact sequence of flat $B$-modules and so, for a $B$-module $N$, by applying the functor $- \otimes_B N$, one obtains an exact sequence
\begin{align*}
	\cdots \to B \otimes_k B \otimes_k N  &\to \begin{aligned}[t] B \otimes_k N 	&\xrightarrow{\phantom{\hspace{3cm}}} N \to 0
	\\b \otimes n &\xmapsto{\phantom{\hspace{3cm}}} bn
	\end{aligned}
		\\ b_0 \otimes b_1 \otimes n &\mapsto b_0 b_1 \otimes n - b_0 \otimes b_1 n	
\end{align*}

Now, consider this exact sequence as an exact sequence of $B$-modules via the homomorphism $i_1 \colon B \to B \otimes_k B, \mkern9mu i_1(b) = b \otimes 1$. This yields that 
\[ \cdots \to B \otimes_k B \otimes_k N \to B  \otimes_k N\]
is a free resolution of the $B$-module $N$, known as bar resolution. In this section, the connecting homomorphisms of $\Tor$ long exact sequences, including the one involved in the definition of $\psi$, will be computed with bar resolutions. 

\subsection*{The homomorphism from Andr\'e-Quillen cohomology to Hochschild cohomology} \hfill \break
Following the preceding notation, we define Hochschild cohomology as
\[ \Hochs^{\ast}(B | k) = \Ext_{B \otimes_k B}^{\ast}(B,B).\]
The elements of $\Hochs^2(B |k)$ can be identified with classes of extensions (for the obvious equivalence relation)
\[ 0 \to I \hookrightarrow E \xrightarrow{\varepsilon} B \to 0\]
where $\varepsilon$ is a homomorphism of (non-necessarily commutative) $k$-algebras which is split as a homomorphism of $k$-modules, $I^2=0$ and $I \cong B$ as $B$-modules (see for instance \cite[Chapter XIV, \S 2]{CartanEilenberg}).

As a result, we have a canonical homomorphism
\[ \gamma \colon \Exal_k(B,B) \to \Hochs^2(B | k)\]
taking the class of an extension in $\Exal_k(B,B)$ to the class in $\Hochs^2(B | k)$ of that same extension.

\begin{proof}[Proof of Theorem \ref{commSnashallSolberg}]
	\hfill \break
	\emph{Step 1}. The homomorphism $\Exal_k(B,B) \xrightarrow{\gamma} \Hochs^2(B | k) \xrightarrow{s} \Ex_{B}^2(N,N)$. \hfill \break
	Let $N$ be a $B$-module and $s \colon \Hochs^2(B | k) \to \Ex_{B}^2(N,N)$ the homomorphism defined above. We are going to compute $s \circ \gamma$:
	
	Let 
	\[[\xi] = \left[ 0 \to B \xrightarrow{i} E \xrightarrow{\varepsilon} B \to 0\right] \in \Exal_k(B,B)\]
	
	Let $\sigma \colon B \to E$ be a $k$-module section of $\varepsilon$ and let $\pi \colon E \to B, \break \mkern9mu \pi(e)= i^{-1}(e - \sigma \varepsilon(e))$, so that $ \pi i = \mathrm{id}_B, \mkern9mu \varepsilon \sigma = \mathrm{id}_B$.
	
	As mentioned before, $\gamma[\xi]$ is the class in $\Hochs^2(B | k)$ of the extension
	\[ \left[0 \to B \to E \to B \to 0\right]\]
	but one can also think of $\gamma[\xi]$ as the class of a cocycle in the complex
	\[ \HHom_{B \otimes_k B}(\mathcal{B}_{\ast}(B), B)\] 
	To see this, let $f \colon B \otimes_k B \to B$ be the $k$-module homomorphism
	\[f(b_1 \otimes b_2) \eqdef i^{-1}(\sigma(b_1)\sigma(b_2) - \sigma(b_1 b_2))\]
	and let $\tilde{f}$ be the image of $f$ under the isomorphism
	\begin{align*}
		\HHom_k(B \otimes_k B, B)& \xrightarrow{\cong} \HHom_{B \otimes_k B}(B \otimes_k (B \otimes_k B) \otimes_k B, B)
		\\ b f(b_1 \otimes b_2) b'& \longleftrightarrow \tilde{f}(b \otimes b_1 \otimes b_2 \otimes b')
	\end{align*}
	
	We will show that, for the resolution $\mathcal{B}_{\ast}(B) \to B$ of the $ B \otimes_k B$-module $B$ 
	\begin{center}
		\begin{tikzcd}
			\cdots \arrow{r} & B^{\otimes 5} \arrow{r}{d_3} & B^{\otimes 4} \arrow{r}{d_2} \arrow{d}{\tilde{f}} & B^{\otimes 3} \arrow{r}{d_1} & B^{\otimes 2} \arrow{r}{d_0} & B \arrow{r} & 0
			\\ & & B& & & &
		\end{tikzcd}
	\end{center}
	$\tilde{f}$ is a cocycle, that is, $\tilde{f} d_3 = 0$. Despite this being well known, we outline the proof below, since we will need some of the computations for later.
		\small
		
		\vspace{0.3cm}
		For any $\left[0 \to B' \xrightarrow{i} E \xrightarrow{\varepsilon} B \to 0\right] \in \Exal_k(B,B')$ (here $B'=B$ as $B$-modules, but we will sometimes use $B'$ in order not to confuse the two $B$s) with a $k$-module section $\sigma$ of $\varepsilon$, one has a $k$-module isomorphism 
		\[ \alpha \colon B' \times B \to E, \mkern9mu \alpha(b',b) = i(b') + \sigma(b),\]
		with inverse $e \mapsto (\pi(e), \varepsilon(e))$. If we define in $B' \times B$ the product 
		\[(b'_1, b_1) \cdot (b'_2, b_2) \eqdef (b_1 b'_2 + b_2 b'_1 + f(b_1 \otimes b_2), b_1 b_2 )\]
		$\alpha$ preserves this product and, as a consequence, it is associative. The following particular case
		\[ \left[(0,b_1) (0, b_2) \right] (0, b_3) = (0,b_1 ) \left[(0,b_2) (0, b_3) \right]\]
		implies the equality
		\begin{align*} \label{eqAst} \tag{$\ast$}
			&i^{-1}\left[\sigma(b_1) \sigma(b_2) - \sigma(b_1 b_2)\right] b_3 - b_1 \cdot i^{-1} \left[ \sigma(b_2) \sigma(b_3) - \sigma(b_2 b_3)\right] \\ &+ i^{-1} \left[ \sigma(b_1 b_2) \sigma(b_3) - \sigma(b_1) \sigma(b_2 b_3)\right]  = 0
		\end{align*}
		and, from this equality, it follows easily that $\tilde{f}d_3(b_0 \otimes b_1 \otimes b_2 \otimes b_3 \otimes b_4)=0$.
		\normalsize
		\vspace{0.3cm}
		
	As a consequence, we have a commutative diagram with exact rows defining $W$
	\begin{center}
		\begin{tikzcd}
			0 \arrow{r} & \ker(d_1) \arrow{r} \arrow{d}{\beta} & B \otimes_k B \otimes_k B \arrow{r}{d_1} \arrow{d} &  B \otimes_k B \arrow{r}{d_0} \arrow[equal]{d} &  B \arrow{r} \arrow[equal]{d} & 0
			\\ 0 \arrow{r} & B \arrow{r} & W \arrow{r} &  B \otimes_k B \arrow{r} &  B \arrow{r} & 0
			
			\arrow[from=2-3, to=1-2, "\blacksquare"{anchor=center, pos=0.125, rotate=180}, draw=none]
		\end{tikzcd}
	\end{center}
	in which $\beta$ is the factorization of $\tilde{f}$ through $\coker(d_3)=\ker(d_1)$:
	\begin{center}
		\begin{tikzcd}
			B^{\otimes 4} \arrow[two heads]{r}{d_2} \arrow[swap]{d}{\tilde{f}} & \ker(d_1) \arrow{dl}{\beta}
			\\ B 
		\end{tikzcd}
	\end{center}
	
	Using $\eqref{eqAst}$ it is easy to see that the $B \otimes_k B$-module homomorphism $\beta$ is the restriction to $\ker(d_1)$ of the additive map $B \otimes_k B \otimes_k B \to B$ given by $b_0 \otimes b_1 \otimes b_2 \mapsto i^{-1}(\sigma(b_0) \sigma(b_1) - \sigma(b_0 b_1)) b_2$.

	The equivalence class of the above extension,
	\[ 0 \to B \to W \to B \otimes_k B \to B \to 0,\]
	in $\Ex_{B \otimes_k B}^2(B,B)$ is the one corresponding with $\gamma[\xi]$. It does not depend on the choice of $\sigma$, since if $\tilde{f}'$ is the cocycle associated with another section $\sigma'$, $\tilde{f}- \tilde{f}'$ is a coboundary. Applying $s$ to this element now yields an extension (recall that we consider the above extension as a sequence of $B$-modules via $i_2 \colon B \to B \otimes_k B$ and apply $- \otimes_B N$)
	\[ 0 \to N \to W \otimes_B N \to B \otimes_k N \to N \to 0 \] 
	which, as an extension of $B$-modules where the $B$-module structure for $W \otimes_B N$ is given by that of $W$ (obtained via $i_1 \colon B \to B \otimes_k B$), while in $B \otimes_k N$ the $B$-module structure is given by the left factor, represents the equivalence class of $s(\gamma[\xi]) $ in $\Ex_B^2(N,N)$.
	
	If $X \xrightarrow{i} Y$, $X \xrightarrow{j} Z $ are $B$-module homomorphisms and $N$ is a $B$-module, $(Y \oplus_X Z) \otimes_B N \cong (Y \otimes_B N) \oplus_{X \otimes_B N} (Z \otimes_B N)$. Then, we can alternatively think of $s(\gamma[\xi]) $ as the bottom row of the following diagram
	\begin{center}
		\begin{tikzcd}[scale cd=0.88, column sep=0.7cm]
			0 \arrow{r} & \ker(d_1) \otimes_B N \arrow{r}{\theta} \arrow{d}{v \eqdef \beta \otimes \mathrm{id}_N} & B \otimes_k B \otimes_k B \otimes_B N \arrow{r}{d_1 \otimes \mathrm{id}_N} \arrow{d} & B \otimes_k B \otimes_B N \arrow{r} \arrow[equal]{d} & B \otimes_B N \arrow{r} \arrow[equal]{d} & 0
			\\0 \arrow{r} & N = B \otimes_B N \arrow{r} & W_N \eqdef W \otimes_B N \arrow{r} & B \otimes_k N \arrow{r} & N \arrow{r} & 0
			
			\arrow[from=2-3, to=1-2, "\blacksquare"{anchor=center, pos=0.125, rotate=180}, draw=none]
		\end{tikzcd}
	\end{center}
	Note that the upper row is exact, since in the exact sequence $ 0 \to  \ker(d_1) \to B \otimes_k B \otimes_k B \to B \otimes_k B \to B \to 0$ the three (non-trivial) rightmost terms are flat, and thus $\ker(d_1)$ is flat too. Additionally, $\ker(d_1) \otimes_B N = \ker(d_1 \otimes_B \mathrm{id}_N)$.
	
	\emph{Step 2}. The homomorphism $\psi \colon \Exal_k(B,B) \to \Ex_B^2(N,N)$. \hfill \break
	The exact sequence $\cdots \to E \otimes_k E \otimes_k E \xrightarrow{d_1^E} E \otimes_k E \xrightarrow{d_0^E} E \to 0 $ gives a short exact sequence
	\[ 0 \to \im(d_1^E)=\ker(d_0^E) \to E \otimes_k E \to E \to 0\]
	and then, by applying $- \otimes_E N$,
	\[ 0 \to \ker(d_0^E) \otimes_E N \to E \otimes_k N \to N \to 0\]
	where $E \otimes_k N$ is a free $E$-module ($E$ acts on the left factor). Using this free presentation of $N$ to compute $\psi([\xi])$, one obtains $\psi([\xi])$ as the bottom row of the diagram
	\begin{center}
		\begin{tikzcd}[column sep=0.7cm]
			0 \arrow{r} & \Tor_1^E(B,N) \arrow{r}{\tilde{\partial}} \arrow{d}{\partial} & B \otimes_E\ker(d_0^E) \otimes_E N \arrow{r}{\chi} \arrow{d} & B \otimes_k N \arrow{r} \arrow[equal]{d} & N \arrow{r} \arrow[equal]{d} & 0
			\\0 \arrow{r} & N = \Tor_0^E(B, N) \arrow{r} & Y \arrow{r} & B \otimes_k N \arrow{r} & N \arrow{r} & 0
			
			\arrow[from=2-3, to=1-2, "\blacksquare"{anchor=center, pos=0.125, rotate=180}, draw=none]
		\end{tikzcd}
	\end{center}

	\emph{Step 3}. Proving that $s \circ \gamma = \psi$. \hfill \break
	We have a commutative diagram
	\begin{center}
		\begin{tikzcd}
			B \otimes_k E \otimes_k N \arrow{r}{r} \arrow{d}{\tilde{d}} & B \otimes_k B \otimes_k N \arrow{d}{d_1 \otimes_B \mathrm{id}_N}
			\\ B \otimes_k N \arrow[equal]{r} & B \otimes_k N
		\end{tikzcd}
	\end{center}
	in which $r$ sends $b \otimes e \otimes n \mapsto b \otimes \varepsilon(e) \otimes n$, and so by the commutativity of this diagram $r$ induces a map from the cycles $q_2 \colon Z(B \otimes_k E \otimes_k N) \eqdef \ker(\tilde{d}) \to \ker(d_1 \otimes_B \mathrm{id}_N) = \ker(d_1) \otimes_B N$, which is surjective because $r$ is as well.
	
	Let 
	\[ q_1 \colon Z(B \otimes_k E \otimes_k N) \to \Tor_1^E(B,N)\]
	be the homomorphism that sends cycles to their equivalence class in homology. 
	
	Since both $q_1$ and $q_2$ are surjective, the preceding pushout diagrams can be ``started'' at $Z(B \otimes_k E \otimes_k N)$, that is, we have the following cocartesian diagrams
	\begin{center}
		\begin{tikzcd}[column sep=small]
			Z(B \otimes_k E \otimes_k N) \arrow{r}{\theta \circ q_2} \arrow{d}{v \circ q_2} & B \otimes_k B \otimes_k N \arrow{d} & & Z(B \otimes_k E \otimes_k N) \arrow{r}{\tilde{\partial} \circ q_1} \arrow{d}{\partial \circ q_1} & B \otimes_E \ker(d_0^E) \otimes_E N \arrow{d}
			\\N \arrow{r} & W_N  & & N \arrow{r} & Y 
			
			\arrow[from=2-2, to=1-1, "\blacksquare"{anchor=center, pos=0.125, rotate=180}, draw=none]
			\arrow[from=2-5, to=1-4, "\blacksquare"{anchor=center, pos=0.125, rotate=180}, draw=none]
		\end{tikzcd}
	\end{center}
	
	We will show that there is a commutative diagram between both pushouts
	\begin{equation} \label{diagPushouts} \tag{CP}
		\begin{tikzcd}[scale cd=0.92]
			Z(B \otimes_k E \otimes_k N) & & B \otimes_k B \otimes_k N & 		
			\\ & Z(B \otimes_k E \otimes_k N) & & B \otimes_E \ker(d_0^E) \otimes_E N 
			\\N & & &	
			\\ & N & & 		
			\arrow[from=1-1, to=1-3, "\theta \circ q_2"] 
			\arrow[from=1-1, to=2-2, "g"]
			\arrow[from=1-1, to=3-1, "v \circ q_2"]
			\arrow[from=1-3, to=2-4, "\tau"]
			\arrow[from=2-2, to=2-4, "\tilde{\partial} \circ q_1"]
			\arrow[from=2-2, to=4-2, "\partial \circ q_1"]
			\arrow[from=3-1, to=4-2, equal]		
		\end{tikzcd}
	\end{equation}
	where $g (b \otimes e \otimes n) = b \otimes \sigma \varepsilon(e) \otimes n$ and  $\tau(b_1 \otimes b_2 \otimes n)= b_1 \otimes (\sigma(b_2) \otimes 1 - 1 \otimes \sigma(b_2)) \otimes n \in B \otimes_E \ker(d_0^E) \otimes_E N$.
	
	%%%%%%%%%%%%%
	So as to show the commutativity of the upper square, we compute, explicitly, $\tilde{\partial} \colon \Tor_1^E(N,B) \to B \otimes_E \ker(d_0^E) \otimes_E N$.
	Take the free resolution of the $E$-module $B$
	\[ \dots \to E \otimes_k E \otimes_k B \to E \otimes_k B \mkern9mu (\to B \to 0), \]
	but write it ``reflected''
	\[ \dots \to B \otimes_k E \otimes_k E \to B \otimes_k E \mkern9mu (\to B \to 0) \]
	so that the maps are defined in the following way
	\begin{align*}
		B \otimes_k E \otimes_k E &\to B \otimes_k E  &B \otimes_k E \to B
		\\b \otimes e_2 \otimes e_1 &\mapsto b \otimes e_1 e_2 - \varepsilon(e_2) b \otimes e_1 &b \otimes e \mapsto \varepsilon(e) b
	\end{align*}
	That is, written in this way, the $E$-module structure lies in the rightmost factor.
	
	Tensoring this resolution over $E$ by the exact sequence
	\[ 0 \to \ker(d_0^E) \otimes_E N \to E \otimes_k N \to N \to 0\]
	(in which the $E$-module structure now lies on the left factor) we obtain the following diagram
	\begin{center}
		\begin{tikzcd}[scale cd=0.8, column sep=tiny]
			& \vdots \arrow{d} & \vdots \arrow{d} & \vdots \arrow{d} & 
			\\0 \arrow{r} & (B \otimes_k E \otimes_k E) \otimes_E (\ker(d_0^E) \otimes_E N) \arrow{r} \arrow{d}{d \otimes  \mathrm{id}_{\ker(d_0^E) \otimes_E N}} & (B \otimes_k E \otimes_k E) \otimes_E (E \otimes_k N) \arrow{r} \arrow{d}{d \otimes  \mathrm{id}_{E \otimes_k N}} & (B \otimes_k E \otimes_k E) \otimes_E N \arrow{r} \arrow{d}{d \otimes \mathrm{id}_N} & 0
			\\ 0 \arrow{r} & (B \otimes_k E) \otimes_E (\ker(d_0^E) \otimes_E N) \arrow{r}  \arrow[two heads]{d} & (B \otimes_k E) \otimes_E (E \otimes_k N) \arrow{r} & (B \otimes_k E) \otimes_E N \arrow{r} & 0
			\\ & B \otimes_E (\ker(d_0^E) \otimes_E N) & & & 
		\end{tikzcd}
	\end{center}
	
	Let $\sum_j b_i \otimes e_j \otimes 1 \otimes n_j$ be a cycle in the upper right term of the diagram. Take the representing element $\sum_j b_j \otimes e_j \otimes 1 \otimes 1 \otimes n_j \in (B \otimes_k E \otimes_k E) \otimes_E (E \otimes_k N)$. After applying the differential, we get
	\[ u = \sum_j b_j \otimes e_j \otimes 1 \otimes n_j - \sum_j \varepsilon(e_j) b_j \otimes 1 \otimes 1 \otimes n_j \]
	
	Since $\sum_j b_j \otimes e_j \otimes 1 \otimes n_j$ was a cycle, by applying $d \otimes \mathrm{id}_N$ we get
	\[ \sum_j b_j \otimes e_j \otimes n_j - \sum_j \varepsilon(e_j) b_j \otimes 1 \otimes n_j = 0 \]
	
	Therefore, in $ B \otimes_k N \mkern9mu (=(B \otimes_k E) \otimes_E N) $ we have
	\[ \sum_j b_j \otimes \varepsilon(e_j) n_j = \sum_j \varepsilon(e_j) b_j \otimes n_j  \]
	and thus we can rewrite $u \in (B \otimes_k E) \otimes_E (E \otimes_k N) = B \otimes_k E \otimes_k N$ as 
	\begin{align*}
		u &= \sum_j b_j \otimes e_j \otimes n_j - \sum_j \varepsilon(e_j) b_j \otimes 1 \otimes n_j 
		\\&= \sum_j b_j \otimes e_j \otimes n_j - \sum_j  b_j \otimes 1 \otimes \varepsilon(e_j) n_j
		\\&= \sum_j b_j \otimes (e_j \otimes n_j -  1 \otimes \varepsilon(e_j) n_j)
	\end{align*}
	and so, in this form, it can be seen as an element of the middle left term of the diagram $(B \otimes_k E) \otimes_E (\ker(d_0^E) \otimes_E N) = B \otimes_k \ker(d_0^E \otimes_E \mathrm{id}_N)$.
	
	As such, the image of this cycle in the 0-th homology of the left column, \break $B \otimes_E \ker(d_0^E \otimes_E \mathrm{id}_N)$, is $\sum_j \varepsilon(e_j) b_j  \otimes n_j - \sum_j  b_j \otimes \varepsilon(e_j) n_j$. This coincides with $ \tilde{\partial} \circ q_1 \circ g (\sum_j b_j \otimes e_j \otimes n_j)$, and it is easy to check that it is precisely what $\tau \circ \theta \circ q_2$ yields as well.

	%%%%%%%%%%%%%
	In order to see the commutativity of the left square of the diagram \eqref{diagPushouts}, we compute $v \circ q_2$ and $\partial \circ  q_1$.
	
	On the one hand, $v \circ q_2 = (\beta \otimes_B \mathrm{id}_N) \circ q_2$ is given by
	\[ b \otimes e \otimes n \mkern6mu(= b \otimes e \otimes 1 \otimes n) \mapsto i^{-1}[\sigma(b) \sigma(\varepsilon(e) - \sigma(b \varepsilon(e))] \cdot 1 \otimes n =  i^{-1}[\sigma(b) \sigma(\varepsilon(e)) - \sigma(b \varepsilon(e))] \cdot n \]
	
	As to $ \partial \circ q_1$, we follow a similar process to the one used before. We consider the free resolution of the $E$-module $N$
	\[ \cdots \to E \otimes_k E \otimes_k N \to E \otimes_k N \mkern9mu (\to N \to 0)\]
	tensored by the exact sequence of $E$-modules 
	\[ 0 \to B \xrightarrow{i} E \xrightarrow{\varepsilon} B \to 0, \]
	which yields the diagram 
	\begin{equation} \label{connectingHom1} \tag{$\ast\ast$}
		\begin{tikzcd}
			& \vdots \arrow{d} & \vdots \arrow{d} & \vdots \arrow{d} & 
			\\0 \arrow{r} & B \otimes_k E \otimes_k N \arrow{r} \arrow{d}{\tilde{d}} & E \otimes_k E \otimes_k N \arrow{r} \arrow{d}{\tilde{d}} & B \otimes_k E \otimes_k N \arrow{r} \arrow{d}{\tilde{d}} & 0
			\\ 0 \arrow{r} & B \otimes_k N \arrow{r} & E \otimes_k N \arrow{r} & B \otimes_k N \arrow{r} & 0
		\end{tikzcd}
	\end{equation}
	An element from $\Tor_1^E(B,N)$ is the class of a cycle $z \in Z(B \otimes_k E \otimes_k N) \break = \ker( \tilde{d} \colon B \otimes_k E \otimes_k N \to B \otimes_k N)$.
	
	By using the diagram above, but starting already from a cycle of the form \break $ \sum_j \varepsilon(e'_j) \otimes \sigma \varepsilon(e_j) \otimes n_j \in B \otimes_k E \otimes_k N$ lifted from $\im(q_1 \circ g)$, we can describe $ \partial \circ q_1$.  Since $ \sum_j \varepsilon(e'_j) \otimes \sigma \varepsilon(e_j) \otimes n_j$ is a cycle, we have that, in $B \otimes_k N$,
	\begin{equation} \label{cycle_q1} \tag{$\ast\ast\ast$}
		\sum_j \varepsilon(e'_j) \varepsilon(e_j) \otimes n_j - \sum_j \varepsilon(e'_j) \otimes \varepsilon(e_j) n_j = 0.
	\end{equation}
	Take a representing element for the cycle, say $ \sum_j e'_j \otimes \sigma \varepsilon(e_j) \otimes n_j \in E \otimes_k E \otimes_k N$. Applying the differential now yields the following in $E \otimes_k N$
	\[  \sum_j e'_j \sigma \varepsilon(e_j) \otimes n_j - \sum_j e'_j \otimes \varepsilon(e_j) n_j\]
	
	Now, this element comes from the element 
	\[ \sum_j( \pi(e'_j \sigma \varepsilon(e_j)) \otimes n_j - \pi(e'_j) \otimes \varepsilon(e_j) n_j)\]
	in $B \otimes_k N$, since by applying $i \otimes_k \mathrm{id}_N$ we obtain
	\begin{align*}
		(i \otimes_k &\mathrm{id}_N) \left(\sum_j( \pi(e'_j \sigma \varepsilon(e_j)) \otimes n_j - \pi(e'_j) \otimes \varepsilon(e_j) n_j)\right) 
		\\&= \sum_j (e'_j \sigma \varepsilon(e_j) - \sigma \varepsilon (e'_j \sigma \varepsilon(e_j))) \otimes n_j - \sum_j (e'_j - \sigma \varepsilon(e'_j)) \otimes \varepsilon(e_j) n_j
		\\&= \sum_j (e'_j \sigma \varepsilon(e_j) \otimes n_j - \sigma \varepsilon (e'_j e_j) \otimes n_j)  - \sum_j (e'_j \otimes \varepsilon(e_j) n_j - \sigma \varepsilon(e'_j) \otimes \varepsilon(e_j) n_j) 
		\\&= \sum_j (e'_j \sigma \varepsilon(e_j) \otimes n_j - e'_j \otimes \varepsilon(e_j) n_j)  - \sum_j (\sigma \varepsilon (e'_j e_j) \otimes n_j - \sigma \varepsilon(e'_j) \otimes \varepsilon(e_j) n_j) 
		\\ &= \sum_j (e'_j \sigma \varepsilon(e_j) \otimes n_j - e'_j \otimes \varepsilon(e_j) n_j),
	\end{align*}
	owing the last equality to \eqref{cycle_q1}.
	
	It follows that the image of the chosen cycle in the 0-th homology of the left column of \eqref{connectingHom1}, $N$, is just
	\begin{align*}
		\sum_j &\left(\pi(e'_j \sigma \varepsilon(e_j)) n_j - \pi(e'_j) \varepsilon(e_j) n_j\right) 
		\\&= \sum_j \left(i^{-1}[e'_j \sigma \varepsilon(e_j) - \sigma \varepsilon(e'_j \sigma \varepsilon(e_j))] n_j - i^{-1}[e'_j - \sigma \varepsilon(e'_j)] \varepsilon(e_j) n_j  \right)
		\\&= \sum_j i^{-1}[e'_j \sigma \varepsilon(e_j) - \sigma \varepsilon(e'_j e_j) - e'_j \sigma \varepsilon(e_j) + \sigma \varepsilon(e'_j) \sigma \varepsilon(e_j)] n_j 
		\\&=\sum_j i^{-1}[\sigma \varepsilon(e'_j) \sigma \varepsilon(e_j) - \sigma \varepsilon(e'_j e_j)] n_j  
	\end{align*}
	
	We have just verified that $\partial \circ q_1  \circ g \left(\sum_j \varepsilon(e'_j) \otimes e_j \otimes n_j\right) = \sum_j i^{-1}[\sigma \varepsilon(e'_j) \sigma \varepsilon(e_j) - \sigma \varepsilon(e'_j e_j)] n_j = v \circ q_2\left(\sum_j \varepsilon(e'_j) \otimes e_j \otimes n_j\right)$, and thus there is an induced map $W_N \to Y$ for which the left square in the following diagram is commutative:
	\begin{center}
		\begin{tikzcd}
			s(\gamma[\xi]) \colon & 0 \arrow{r} & N \arrow{r} \arrow[equal]{d} & W_N \arrow{r} \arrow{d} & B \otimes_k N \arrow{r} \arrow[equal]{d} & N \arrow{r} \arrow[equal]{d} & 0
			\\ \psi([\xi]) \colon & 0 \arrow{r} & N \arrow{r} & Y \arrow{r} & B \otimes_k N \arrow{r} & N \arrow{r} & 0
		\end{tikzcd}
	\end{center}
	Commutativity of the left square is obvious, and commutativity of the central square, using the ``completed'' version of diagram \eqref{diagPushouts}, can be reduced to showing commutativity of the diagram
	\begin{center}
		\begin{tikzcd}
			B \otimes_k B \otimes_k N \arrow{r}{d_1 \otimes_B N} \arrow{d}{\tau} & B \otimes_k N
			\\ B \otimes_E \ker(d_0^E) \otimes_E N \arrow {r}{\chi} & B \otimes_k N \arrow[equal]{u}
		\end{tikzcd}
	\end{center}
	and this is easy.	
\end{proof}

\end{document}